\documentclass[12pt,reqno]{amsart}
\usepackage{amssymb,amsmath,enumerate}
\usepackage{mathrsfs}
\usepackage{color}

\newtheorem{theorem}{Theorem}[section]
\newtheorem{proposition}[theorem]{Proposition}

\theoremstyle{definition}
\newtheorem{definition}[theorem]{Definition}
\newtheorem{example}[theorem]{Example}

\numberwithin{equation}{section}

\newcommand{\clb}{\mathscr{B}}
\newcommand{\bd}{\mathbb{D}}
\newcommand{\cld}{\mathcal{D}}
\newcommand{\cle}{\mathcal{E}}
\newcommand{\clh}{\mathcal{H}}

\newcommand{\clm}{\mathcal{M}}
\newcommand{\clq}{\mathcal{Q}}
\newcommand{\clw}{\mathcal{W}}

\newcommand{\raro}{\rightarrow}
\newcommand{\D}{\mathbb{D}}
\newcommand{\Z}{\mathbb{Z}}

\begin{document}

\title[Hyponormal contractions and analytic shifts]{Hyponormal contractions and analytic shifts}

\author[Sneha B]{Sneha B}
\address{Statistics and Mathematics Unit, Indian Statistical Institute, 8th Mile, Mysore Road, Bangalore, 560059, India}
\email{sneharbkrishnan@gmail.com, rs\textunderscore math2105@isibang.ac.in}

\author[Bala]{Neeru Bala}
\address{IIT(ISM) Dhanbad, Jharkhand, 826004, India}
\email{neerusingh41@gmail.com, neerubala@iitism.ac.in}

\author[Sarkar]{Jaydeb Sarkar}
\address{Statistics and Mathematics Unit, Indian Statistical Institute, 8th Mile, Mysore Road, Bangalore, 560059, India}
\email{jay@isibang.ac.in, jaydeb@gmail.com}
	

	
\begin{abstract}
Hyponormal operators are known to be among the most difficult operators to analyze. In this work, we focus on two finite types of hyponormal operators. The first type becomes analytic shifts, while the second type admits analytic models. A basic model for hyponormal operators plays a key role in our analysis.
\end{abstract}


\maketitle

\tableofcontents

\section{Introduction}

Examples and representations of operators are at the heart of the theory of linear operators on Hilbert spaces. The spectral theorem for normal operators is a prime illustration of this theme. Along this line, subnormal operators are the next relevant objects. While the program of subnormal operators continues to undergo various developments, there are several theories related to them that are, to a large extent, satisfactory (see \cite{Aleman, Mccarthy, Morrel} and the excellent survey \cite{CL}). The next step lies in the analysis of hyponormal operators, which are known to be complex. An operator $T \in \clb(\clh)$ is \textit{hyponormal} if
\[
T T^* \leq T^* T.
\]
Equivalently, in terms of the self-commutator, we have
\[
[T^*,T] := T^*T - T T^* \geq 0.
\]
In this paper, $\clh$ will denote an arbitrary but fixed separable Hilbert space over $\mathbb{C}$, and we write $\clb(\clh)$ as the space of all bounded linear operators on $\clh$. Two other important classes of operators are finite-rank operators and analytic shifts on reproducing kernel Hilbert spaces. In this paper, we consider two types of finite-rank hyponormal operators:

\begin{definition}
Let $T \in \clb(\clh)$ and let $n \in \mathbb{N}$. We say that:
\begin{enumerate}
\item $T$ is a finite-isometry if $T$ is a hyponormal contraction and
\[
\text{rank} (I - T^*T) < \infty.
\]
\item $T$ is $n$-finite if $T$ is a hyponormal contraction and
\[
\text{rank} [T^*,T] = n,
\]
and
\[
\text{ran}[T^*, T] \subseteq \cld_{T^*} \ominus \cld_{T}.
\]
\end{enumerate}
\end{definition}

In this paper, we follow standard notations \cite{NagyFoias}: Given a contraction $T \in \clb(\clh)$ (that is, $\|Th\| \leq \|h\|$ for all $h \in \clh$), the \textit{defect operators} are defined as
\[
D_T=(I-T^*T)^{\frac{1}{2}} \text{ and } D_{T^*}=(I-TT^*)^{\frac{1}{2}}.
\]
The corresponding \textit{defect spaces} are defined as
\[
\cld_T = \overline{\text{span}}D_T \text{ and } \cld_{T^*}=\overline{\text{span}}D_{T^*},
\]
respectively. Among contractions, defect operators provide a way to detect hyponormality: A contraction $T$ is hyponormal if and only if
\[
D_T \leq D_{T^*}.
\]
In particular, we have $\cld_T \subseteq \cld_{T^*}$. It also follows that (see Section \ref{sec: fin rank self com})
\[
\text{ran} [T^*, T] \subseteq \cld_{T^*}.
\]
Moreover, in Proposition \ref{prop: n finite}, we prove that 
\[
\text{dim} (\cld_{T^*} \ominus \cld_T) = n.
\]
These observations provide additional motivation for the study of the class of $n$-finite operators.

For our purposes, a reproducing kernel Hilbert space is a Hilbert space $\clh_k$ of $\cle$-valued ($\cle$ being a Hilbert space) analytic functions on the open unit disc $\D \subset \mathbb{C}$, such that the evaluation operator $ev_w: \clh_k \raro \cle$ (defined by $ev_w(f) = f(w)$ for all $f \in \clh_k$) is continuous for all $w \in \D$. This property yields the construction of the kernel function $k: \D \times \D \raro \clb(\cle)$ as:
\[
k(z,w) = ev_z \circ ev_w^*,
\]
for all $z, w \in \D$. Define the multiplication operator $M_z$ on $\clh_k$ by
\[
M_zf = zf,
\]
for all $f \in \clh_k$. We say that $M_z$ is an \textit{analytic shift} if it is a contraction and left-invertible. In general, $T \in \clb(\clh)$ is said to be analytic if
\[
\bigcap_{n=0}^\infty T^n \clh = \{0\}.
\]
Evidently, $M_z$ on $\clh_k$ is analytic in this sense.

From our perspective, the idea of analytic shifts is motivated by the following well-known fact: If $T \in \clb(\clh)$ is a left-invertible and analytic operator, then $T$ is unitarily equivalent to an analytic shift $M_z$ on some $\clh_k$ \cite{Shimorin}. We express this as
\begin{equation}\label{Shimorin thm}
T \cong M_z.
\end{equation}
This paper presents results around the following three problems:

\begin{enumerate}
\item Model of finite-isometries.
\item Representations of $n$-finite operators.
\item Examples of analytic shifts.
\end{enumerate}

A hyponormal operator is said to be \textit{pure} if it does not have any normal summand; that is, there is no reducing subspace on which the restriction of the operator is normal. In the theory of hyponormal operators, it is standard, and as we will also do, to assume that the operators under consideration are pure hyponormal.

One of the main results of this paper provides a complete characterization of pure finite-isometries (see Theorem \ref{main thm}): Let $T\in\clb(\clh)$ be a pure finite-isometry. Then there exists an analytic shift $M_z$ on some $\clh_k$ such that
\[
T \cong M_z.
\]
This result, like many others in this paper, is based on a basic yet curious model of hyponormal contractions (see Proposition \ref{matrix rep}): Let $T\in\clb(\clh)$ be a pure hyponormal contraction. Then
\[
T \cong
\begin{bmatrix}
S&A\\
0&B
\end{bmatrix} \in \clb((\clh\ominus\cld_T)\oplus\cld_T),
\]
where $B\in\clb(\cld_{T})$ is a contraction, $S\in\clb(\clh\ominus\cld_T)$ is a unilateral shift, and $A\in\clb(\cld_T,\clh\ominus\cld_T)$ such that
\[
S^*A=0.
\]

Recall that a \textit{unilateral shift} $S \in \clb(\clh)$ is an isometry without any unitary summand. The latter property is equivalent to $S$ being analytic. Unilateral shift operators are concrete examples of analytic shifts. In the context of basic models of hyponormal contractions, Theorem \ref{thm 2.3} presents the following examples of analytic shifts: Let $\clh_1$ and $\clh_2$ be Hilbert spaces, $T\in\clb(\clh_1\oplus\clh_2)$, and let $S \in \clb(\clh_1)$ be a unilateral shift. Let
\[
T=\begin{bmatrix}
S&A\\
0&B
\end{bmatrix}.
\]
Assume that $A^*A+B^*B \in \clb(\clh_2)$ is invertible, $B$ is analytic, and $S^*A=0$. Then
\[
T \cong M_z,
\]
for some analytic shift $M_z$ on some $\clh_k$. 

There are plenty of natural examples of $2 \times 2$ operator matrices (such as those involving Toeplitz operators) that satisfy the above conditions. Moreover, the proof of the preceding result is instrumental in providing analytic shift models for finite isometries.

Finally, we turn to $n$-finite operators. Let $T \in \clb(\clh)$ be an $n$-finite operator. There exists an orthonormal set $\{e_j\}_{j=1}^n \subseteq \cld_{T^*} \ominus \cld_T$ and scalars $\{\alpha_j\}_{j=1}^n \subseteq (0,1]$ such that (note that $[T^*, T]$ is a self-adjoint operator)
\[
[T^*, T] = \sum_{j=1}^n \alpha_j e_j \otimes e_j.
\]
In general, for each $u, v$ in $\clh$, we define the rank-one operator $u\otimes v$ by $(u\otimes v)h = \langle h, v \rangle u$ for all $h\in\clh$.

We adopt the following notational convention: given a general contraction $T \in \clb(\clh)$ with the self-commutator representation described above, we define $m \times m$ diagonal matrix $D_{\{\alpha_i\}_{i=1}^m} \in\clb(\mathbb{C}^m)$ by
\[
D_{\{\alpha_i\}_{i=1}^m} =
\begin{bmatrix}
\sqrt{1-\alpha_1} & 0 & \cdots & 0\\
0 & \sqrt{1-\alpha_2} & \cdots & 0\\
\vdots & \vdots & & \vdots\\
0 & 0 & \cdots & \sqrt{1-\alpha_m}\\
\end{bmatrix}.
\]
Moreover, given a natural number $n \geq m$, we define $\tilde{D}_{\{\alpha_i\}_{i=1}^m}: \mathbb{C}^m \raro \mathbb{C}^n$ by
\[
\tilde{D}_{\{\alpha_i\}_{i=1}^m} = \begin{bmatrix} D_{\{\alpha_i\}_{i=1}^m} \\0_{\mathbb{C}^{n-m}}\end{bmatrix}.
\]

In Theorem \ref{main result2}, we prove the following: Let $T \in \clb(\clh)$. Then $T$ is  $n$-finite, with
\[
[T^*, T] = \sum_{j=1}^n \alpha_j e_j \otimes e_j,
\]
for some orthonormal set $\{e_j\}_{j=1}^n \subseteq \cld_{T^*} \ominus \cld_T$ and scalars $ \{\alpha_j\}_{j=1}^n \subseteq (0,1]$, if and only if
$T \cong X_1$ or $T \cong X_2$, where
\[
X_1=\begin{bmatrix}
M_z\otimes I_{\mathbb{C}^m}&0\\
0&N
\end{bmatrix} \in \clb((H^2(\bd)\otimes\mathbb{C}^n)\oplus \clh_0),
\]
and
\[
X_2=\begin{bmatrix}
M_z\otimes I_{\mathbb{C}^m} &P_{\mathbb{C}}\otimes \tilde{D}_{\{\alpha_i\}_{i=1}^m} & 0
\\
0&M_z^*\otimes {D}_{\{\alpha_i\}_{i=1}^m} & 0
\\
0&0&N
\end{bmatrix} \in \clb((H^2(\bd)\otimes\mathbb{C}^n)\oplus(H^2(\bd)\otimes\mathbb{C}^m)\oplus \clh_0),
\]
and $N$ is a normal operator acting on a closed subspace $\clh_0\subseteq\clh$. In the above, $H^2(\D)$ denotes the Hardy space (see the following section for more details), and $P_\mathbb{C}$ denotes the orthogonal projection of $H^2(\D)$ onto the subspace of constant functions.

Hyponormal operators with finite-rank self-commutators have been studied extensively. For instance, Xia \cite{Xia} proved that if the range of the self-commutator of a hyponormal operator is rank one and invariant under  its conjugate, then the operator is a linear combination of the identity operator and the unilateral shift. For results concerning models of hyponormal operators with rank one or two self-commutators and the invariance of the kernels of the self-commutators, we refer the reader to Morrel \cite{Morrel} and Lee and Lee \cite{Lee} (also see \cite{Xia, Xia1, DY}). In \cite{SDas}, Das studied the intricate structure of a class of completely non-unitary contractions, in particular hyponormal contractions, under the assumption that both defect spaces are finite-dimensional.

The paper is divided into four sections, excluding the present one. In Section \ref{sec: basic model}, we present a basic model for hyponormal contractions. Section \ref{sec: shift} discusses examples of analytic shifts related to the basic models introduced in the previous section. In Section \ref{sec: fr hypo}, we prove that finite-isometries are unitarily equivalent to analytic shifts. Section \ref{sec: fin rank self com}, presents analytic models for $n$-finite operators. The final section, Section \ref{sec: eg}, presents several examples and illustrations of the theory and concepts developed in this paper.

\section{A basic model}\label{sec: basic model}

The goal of this section is to present a basic model that represents all pure hyponormal contractions. We use the classical  \textit{Nagy-Foias and Langer orthogonal decompositions} of contractions \cite{HL, NF60}: Given a contraction $T \in \clb(\clh)$, define closed subspaces $\clh_u$ and $\clh_{c}$ of $\clh$ by
\[
\clh_u = \{h \in \clh: \|T^m h\| = \|T^{*m} h\| = \|h\|, m \in \Z_+\},
\]
and
\[
\clh_{c} = \clh_u^\perp,
\]
respectively. Then both $\clh_u$ and $\clh_{c}$ reduce $T$, and
\[
\clh = \clh_u \oplus \clh_{c}.
\]
Moreover, $T|_{\clh_u}$ is unitary whereas $T|_{\clh_{c}}$ is completely non-unitary (that is, $T|_{\clh_{c}}$ does not have unitary summand). Therefore, we have the decomposition of $T$ on $\clh = \clh_u \oplus \clh_{c}$ as
\[
T = \begin{bmatrix}
T|_{\clh_u} & 0
\\
0 & T|_{\clh_{c}}
\end{bmatrix}.
\]

Given a hyponormal contraction $T \in \clb(\clh)$, in view of $D_T \leq D_{T^*}$, we know that $\cld_T \subseteq \cld_{T^*}$. Moreover, in what follows, we will use the orthogonal decomposition
\[
\clh=(\clh\ominus\cld_T)\oplus\cld_T.
\]
With this preparation, we now turn to the model for hyponormal contractions, which will be used extensively throughout the paper.

\begin{proposition}\label{matrix rep}
Let $T\in\clb(\clh)$ be a pure hyponormal contraction. Then
\begin{align}\label{eqn 11}
T \cong
\begin{bmatrix}
S&A\\
0&B
\end{bmatrix} \in \clb((\clh\ominus\cld_T)\oplus\cld_T),
\end{align}
where $B\in\clb(\cld_{T})$ is a contraction,  $S\in\clb(\clh\ominus\cld_T)$ is a unilateral shift, and $A\in\clb(\cld_T,\clh\ominus\cld_T)$ such that
\[
S^*A=0.
\]
\end{proposition}
\begin{proof}
Since $T$ is a hyponormal operator, we know that $\cld_T\subseteq\cld_{T^*}$. Consequently
\[
\clh\ominus\cld_{T^*}\subseteq\clh\ominus\cld_T.
\]
An elementary fact about contractions implies that $T|_{\clh\ominus\cld_T}:\clh\ominus\cld_T\rightarrow\clh\ominus\cld_{T^*}$ is a unitary operator with inverse $T^*|_{\clh\ominus\cld_{T^*}}:\clh\ominus\cld_{T^*}\rightarrow\clh\ominus\cld_T$ \cite[Page 8]{NagyFoias} (in this context, it is useful to note that $T \cld_T \subseteq \cld_{T^*}$). In view of $\clh\ominus\cld_{T^*}\subseteq\clh\ominus\cld_T$, we consider $T|_{\clh\ominus\cld_T}$ as an operator on $\clh\ominus\cld_T$, that is, $T|_{\clh\ominus\cld_T} \in \clb(\clh\ominus\cld_T)$. Define $S\in\clb(\clh\ominus\cld_T)$ by
\[
Sx=Tx,
\]
for all $x\in\clh\ominus\cld_T$. Therefore,
\[
S = T|_{\clh\ominus\cld_T},
\]
defines an isometry on $\clh\ominus\cld_T$. With respect to the decomposition $\clh=(\clh\ominus\cld_T)\oplus\cld_T$, we write
\[
T=\begin{bmatrix}
S&A\\
0&B
\end{bmatrix}.
\]
Since $S$ on $\clh\ominus\cld_T$ is an isometry, the von Neumann and Wold decomposition theorem \cite[page 3]{NagyFoias} guarantees the existence of closed subspaces $\clw_u$ and $\clw_s$ of $\clh\ominus\cld_T$ such that
\begin{enumerate}
\item $\clw_u\oplus \clw_s = \clh\ominus\cld_T$.
\item Both $\clw_u$ and $\clw_s$ reduce $S$.
\item $U:=S|_{\clw_u}$ is a unitary operator on $\clw_u$.
\item  $S_0:=S|_{\clw_s}$ is a unilateral shift on $\clw_s$.
\end{enumerate}
Therefore, we have the representation of $T$ as
\[
T=\begin{bmatrix}
U&0&A_1\\
0&S_0&A_2\\
0&0&B
\end{bmatrix},
\]
on $\clh= \clw_u \oplus \clw_s \oplus\cld_T$. We claim that $A_1 = 0$ (note that $A_1: \cld_T \raro \clw_u$). Indeed, for $x\in \clw_u$, we have
\begin{align*}
\langle (U^*U-UU^*-A_1A_1^*)x,x\rangle=\langle(T^*T-TT^*)x,x\rangle\geq 0.
\end{align*}
Since $U$ is a unitary operator, we have $U^*U - U U^* = 0$, which yields
\[
-\langle A_1A_1^*x,x\rangle\geq 0,
\]
for every $x\in W_u$. This implies $A_1=0$, and consequently
\[
T=\begin{bmatrix}
U&0&0
\\
0&S_0&A_2
\\
0&0&B
\end{bmatrix}.
\]
Since $T$ is a pure hyponormal operator, it follows that $W_u = \{0\}$ (note that $\clw_u$ reduces $T$ and $T|_{\clw_u}$ is unitary); therefore, $S=S_0$ is a unilateral shift. As a result, we have
\[
T=\begin{bmatrix}
S & A_2\\
0 & B
\end{bmatrix},
\]
which also implies
\[
T^*T-TT^*=\begin{bmatrix}
P_{\ker S^*}-A_2A_2^*&*
\\
*&*
\end{bmatrix}.
\]
Since $T$ is a hyponormal operator, we have
\[
P_{\ker S^*} - A_2A_2^*\geq 0.
\]
By Douglas' range inclusion theorem,
\[
\text{ran} A_2 \subseteq \ker S^*,
\]
which is equivalent to saying that $S^* A_2=0$. This completes the proof of the result.
\end{proof}

Recall that the \textit{multiplicity} of a unilateral shift $S$ is the number
\[
d = \text{dim} \ker S^* \in \mathbb{N} \cup \{\infty\}.
\]
The \textit{Hardy space} over $H^2(\D)$ is the Hilbert space of all analytic functions on $\D$ whose power series coefficients are square-summable. Note that $M_z$ on $H^2(\D)$ is a unilateral shift of multiplicity $1$ (as $\dim(\ker M_z^*) = 1$). A unilateral shift of multiplicity $d \in \mathbb{N} \cup \{\infty\}$ is unitarily equivalent to $M_z$ on a $\cld$-valued Hardy space $H^2_\cld(\D)$, where $\cld$ is a Hilbert space and
\[
\text{dim} \cld = d.
\]
In particular, in the previous result, $S$ on $\clh \ominus \cld_T$ can be replaced by $M_z \otimes I_\cle$ on $H^2(\D) \otimes \cle$ for some Hilbert space $\cle$ such that
\[
\dim \cle = \dim(\clh \ominus \cld_T).
\]

The following elementary fact is in the spirit of the operator $T$ considered in the above proposition (here we are assuming $B = S^*$). This will be used again in what follows.

\begin{proposition}\label{proposition 1}
Let $S\in\clb(\clh)$ be an isometry and let $X\in\clb(\clh)$. Define $T\in\clb(\clh\oplus\clh)$ by
\[
T = \begin{bmatrix}
S&X\\
0&S^*
\end{bmatrix}.
\]
Then $T$ is a hyponormal contraction if and only if $X$ is a partial isometry with
\[
\text{ran} X \subseteq \text{ran} X^* = \ker S^*.
\]
Moreover, in this case, $T$ is an isometry.
\end{proposition}
\begin{proof}
First, we compute
\[
T^*T=\begin{bmatrix}
I&S^*X
\\\
X^*S& X^*X+SS^*
\end{bmatrix},
\]
and
\[
TT^*=\begin{bmatrix}
SS^*+XX^*&XS\\
S^*X^*& I
\end{bmatrix}.
\]
Assume that $T$ is a hyponormal contraction. Then  $S^*X=0$ and $XS=0$. Equivalently, $\text{ran} X,\,\text{ran} X^* \subseteq \ker S^*$. Moreover, we have
\begin{align*}
I-T^*T=\begin{bmatrix}
0&0\\
0& P_{\ker S^*}-X^*X
\end{bmatrix}, I-TT^*=\begin{bmatrix}
P_{\ker S^*}-XX^*&0\\
0&0
\end{bmatrix}.
\end{align*}
Note that $I- T^* T \geq 0$ gives $X^*X \leq P_{\ker S^*}$. On the other hand, $T^* T \geq T T^*$ implies $X^*X+SS^* \geq I$, or equivalently, $X^*X \geq P_{\ker S^*}$. Therefore, we conclude that
\[
X^*X = P_{\ker S^*}.
\]
By using $XS=0$, we find
\[
XX^*X=XP_{\ker S^*} = X(I-P_{\text{ran} S})=X,
\]
and hence, $X$ is a partial isometry with $\text{ran} X^* = \text{ran} (X^*X) = \ker S^*$.

\noindent Conversely, assume that $X$ is a partial isometry with $\text{ran} X\subseteq \text{ran} X^* =   \ker S^*$. Then $S^*X=0$, $XS=0$, and $X^*X=P_{\text{ran} X^*}=P_{\ker S^*}$. Therefore, we have
\[
T^*T-TT^*=\begin{bmatrix}
P_{\text{ker}(S^*)}-XX^*&0\\
0& 0
\end{bmatrix}.
\]
Clearly, $T$ is hyponormal and  $I-T^*T=0$. This completes the proof of the proposition.
\end{proof}

As pointed out earlier, Das \cite{SDas} analyzed the structure of a class of hyponormal contractions whose defect spaces are both finite-dimensional. One of the basic tools used in \cite{SDas} is the classical Nagy-Foias and Langer decompositions of contractions, which is also used in this section in constructing the basic model for pure hyponormal contractions.

\section{Analytic shifts}\label{sec: shift}

In the previous section, we constructed a model for pure hyponormal operators. In this section, we use that model as motivation to present examples of analytic shifts.

The following result serves two purposes: first, it provides examples of analytic shifts from a more general operator-theoretic perspective; second, it lays the groundwork for showing that finite-isometries are analytic shifts, a result we will explore in the next section. Although the invertibility assumptions imposed on the operator in what follows may appear somewhat artificial, we will give natural existence of such operators. We recall that if $T \in \clb(\clh)$ is left-invertible and analytic, then there exists an analytic shift $M_z$ on some reproducing kernel Hilbert space $\clh_k$ such that (see \eqref{Shimorin thm})
\[
T \cong M_z.
\]

\begin{theorem}\label{thm 2.3}
Let $\clh_1$ and $\clh_2$ be Hilbert spaces, $T\in\clb(\clh_1\oplus\clh_2)$, and let $S \in \clb(\clh_1)$ be a unilateral shift. Let
\[
T=\begin{bmatrix}
S&A\\
0&B
\end{bmatrix}.
\]
Assume that $A^*A+B^*B \in \clb(\clh_2)$ is invertible, $B$ is analytic, and
\[
S^*A=0.
\]
Then $T$ is unitarily equivalent to an analytic shift.
\end{theorem}
\begin{proof}
Set $\clh = \clh_1 \oplus \clh_2$. As $S^*A=0$, it follows that
\[
T^*T=\begin{bmatrix}
I&0\\
0&A^*A+B^*B
\end{bmatrix}.
\]
Since $A^*A+B^*B$ is invertible, it follows that $T^*T$ is invertible, and consequently, $T$ is left-invertible. We turn to the nontrivial part of the proof: showing that $T$ is analytic, that is,
\[
\underset{n\geq 0}{\bigcap}T^n\clh=\{0\}.
\]
On the contrary, assume that there exists a nonzero vector $\begin{bmatrix}h\\ g \end{bmatrix} \in \clh$ such that
\[
 \begin{bmatrix}
			h\\
			g
\end{bmatrix}\in\underset{n\geq 0}{\bigcap}T^n\clh.
\]
For each $n \in \Z_+$, there exist $C_n \in \clb(\clh_2, \clh_1)$ and $\begin{bmatrix}h_n\\g_n\end{bmatrix}\in\clh$ such that
\[
T^n= \begin{bmatrix}
S^n&C_n\\
0&B^n
\end{bmatrix},
\]
and
\[
T^n\begin{bmatrix}
h_n\\
g_n
\end{bmatrix}=\begin{bmatrix}
h\\
g
\end{bmatrix}.
\]
Assume for a moment that $g_n=0$ for all $n\in\Z_+$. Then, on one hand,
\[
T^n \begin{bmatrix}h_n\\ g_n \end{bmatrix} = \begin{bmatrix} S^n h_n\\ 0 \end{bmatrix},
\]
and on the other hand,
\[
T^n \begin{bmatrix}
h_n\\g_n
\end{bmatrix} = \begin{bmatrix}
S^n&C_n\\0&B^n
\end{bmatrix}\begin{bmatrix}
h_n\\0
\end{bmatrix}=\begin{bmatrix}
h\\g
\end{bmatrix}.
\]
Comparing the above identitites, we have $g=0$ and $h = S^n h_n$ for all $n \in \Z_+$. Since $S$ is analytic, we conclude that $h=0$, and hence we arrive at the contradiction that
\[
\begin{bmatrix}
h\\g
\end{bmatrix} = \begin{bmatrix}
0\\0
\end{bmatrix}.
\]
In view of the above observation, we now assume that there exists some $m_0\in\Z_+$ such that $g_{m_0}\ne 0$ and
\[
g_m=0,
\]
for all $m<m_0$. We also have
\begin{align*}\begin{bmatrix}
h\\g
\end{bmatrix} =
\begin{bmatrix}
S&A\\0&B
\end{bmatrix}\begin{bmatrix}
h_1\\g_1
\end{bmatrix}=\begin{bmatrix}
S^2&SA+AB\\
0&B^2
\end{bmatrix}\begin{bmatrix}
h_2\\g_2
\end{bmatrix}=\cdots.\end{align*}
Comparing the components, we find
\begin{equation}\label{eqn3}
h=Sh_1 + Ag_1 = S^2h_2 + (SAg_2+ABg_2) = \cdots,
\end{equation}
and
\begin{equation}\label{eqn4}
g =B^n g_ n,
\end{equation}
for all $n \in\Z_+$. We claim that
\begin{equation}\label{induction equation}
g_m=B^ng_{n+m},
\end{equation}
for all $m,n\ge 1$. We prove the claim by induction on $m\in\mathbb{N}$. Since $S^*A=0$, we have
\[
\text{ran} A \perp \text{ran} S.
\]
It follows that
\[
Ag_1=ABg_2=AB^2g_3=\cdots.
\]
By considering $Ag_1=ABg_2$ from the above, we find
\[
A(g_1 - B g_2) = 0,
\]
and, on the other hand, \eqref{eqn4} for $n=1,2$, implies that $B(g_1 - B g_2) = 0$. Summarizing, we have
\[
B(g_1-Bg_2)=A(g_1-Bg_2)=0.
\]
Similarly, we have
\[
B(g_1-B^2g_3)=A(g_1-B^2g_3)=0,
\]
and so on. In other words, we have
\[
g_1 - B^n g_{n+1}\in \ker B \cap \ker A,
\]
for every $n\in\mathbb{N}$. However, since $A^*A+B^*B$ is invertible, it follows that
\[
\ker B \cap \ker A = \{0\},
\]
and as a result
\[
g_1-B^ng_{n+1}=0,
\]
for every $n\in\mathbb{N}$. Therefore, we have
\[
g_1=Bg_2=B^2g_3=\cdots,
\]
which proves \eqref{induction equation} for $m=1$. Next, we assume that \eqref{induction equation} holds for $m-1$; that is,
\[
g_{m-1}=B^ng_{n+m-1},
\]
for all $n\ge 1$. By \eqref{eqn3}, we have
\[
h = S^{m-1} h_{m-1} + S^{m-2} A g_{m-1} + AB g_{m-1} = S^{m} h_{m} + S^{m-1} A g_{m} + AB g_{m} = \cdots.
\]
Applying $S^{*{m-1}}$ to either side of the above, and using the fact that $\text{ran} A \perp \text{ran} S$, we obtain
\begin{align*}
S^{*{m-1}}h  = &h_{m-1}
\\
= & Sh_m+Ag_m
\\
=& S^2h_{m+1}+SAg_{m+1}+ABg_{m+1}
\\
=& S^3h_{m+2}+S^2Ag_{m+2}+SABg_{m+2}+AB^2g_{m+2}
\\
&\vdots
\end{align*}
Again, using the orthogonality property $\text{ran} A \perp \text{ran} S$, we have
\[
Ag_m=ABg_{m+1}=AB^2g_{m+2}=\cdots,
\]
which yields
\[
A(g_m-Bg_{m+1})=A(g_m-B^2g_{m+2})=A(g_m-B^3g_{m+3})=\cdots=0.
\]
Now, as we know by the induction hypothesis,
\[
g_{m-1}=Bg_m=B^2g_{m+1}=B^3g_{m+2}=\cdots,
\]
it follows that
\[
B(g_m-Bg_{m+1})=B(g_m-B^2g_{m+2})=B(g_m-B^3g_{m+3})=\cdots=0,
\]
and hence, for every $ n\in \mathbb{N}$, we have
\[
g_m-B^ng_{m+n}\in \ker B \cap \ker A = \{0\},
\]
which implies that \eqref{induction equation} holds for all $m$. As a result, we have
\begin{equation}\label{eqn: g_m}
g_m\in\underset{n\geq 0}{\bigcap}B^n\clh_2,
\end{equation}
for all $m\in\mathbb{N}$. In particular, $g_{m_0}\in\underset{n\geq 0}{\cap}B^n\clh_2$, a contradiction to the fact that
\[
\underset{n\geq 0}{\bigcap}B^n\clh_2=\{0\}.
\]
This proves that $T$ is analytic. Since we have already shown that $T$ is left-invertible, we can finally conclude that $T$ is unitarily equivalent to an analytic shift.
\end{proof}

In the final section, we present examples that satisfy the conditions of the preceding theorem.

\section{Finite-isometries}\label{sec: fr hypo}

In this section, we completely characterize pure finite-isometries. Part of the proof follows the approach used in Theorem \ref{thm 2.3}.

\begin{theorem}\label{main thm}
Let $T\in\clb(\clh)$ be a pure finite-isometry. Then $T$ is unitarily equivalent to an analytic shift.
\end{theorem}
\begin{proof}
We first show that the point spectrum of $T$ is empty, that is, $\sigma_p(T)=\emptyset$. On the contrary, assume that $\lambda\in\sigma_p(T)$. Since $T-\lambda I$ is a hyponormal operator, it follows that
\[
\ker(T-\lambda I)\subseteq \ker(T^*-\overline{\lambda} I).
\]
Clearly, $\ker(T-\lambda I)$ is invariant under $T$. We claim that it also reduces $T$. Indeed, let $f \in \ker(T-\lambda I)$. In view of the above inclusion, it follows that
\[
T(T^*f) = \overline{\lambda} Tf = |\lambda|^2 f,
\]
which completes the proof of the claim. In particular, we have
\[
T|_{\ker(T-\lambda I)} = \lambda I_{\ker(T-\lambda I)},
\]
is a normal operator. However, by assumption, $T$ is a pure hyponormal operator. This leads to a contradiction unless $\sigma_p(T)=\emptyset$.

\noindent Following Proposition \ref{matrix rep}, we now write
\[
T=\begin{bmatrix}
S&A\\
0&B
\end{bmatrix},
\]
on $\clh=(\clh\ominus\cld_T)\oplus\cld_T$, where $S$ is a unilateral shift and $S^*A=0$. By assumption, there exists a finite-rank operator $F\in\clb(\clh)$ such that
\[
T^*T=I+F.
\]
In particular, $T^*T$ is a self-adjoint Fredholm operator with $\ker(T^*T)=\ker T = \{0\}$. We conclude that $T^*T$ is an invertible operator. Therefore, $T$ is left-invertible. Moreover, since
\[
T^*T=\begin{bmatrix}
I&0\\
0&A^*A+B^*B
\end{bmatrix},
\]
it follows that $A^*A+B^*B$ is invertible.

\noindent Next, we observe that no eigenvector of $B$ lies in the null space of $A$. Indeed, if there exist $\lambda\in\mathbb{C}$ and nonzero $h \in \clh$ such that $(B-\lambda I)h=0$ and $Ah=0$, then
\[
T\begin{bmatrix}0\\h
\end{bmatrix}=\lambda\begin{bmatrix}0\\h
\end{bmatrix},
\]
which contradicts the fact that $\sigma_p(T)=\emptyset$. We are now ready to begin the proof that $T$ is an analytic operator. Part of the proof follows the lines of the proof of Theorem \ref{thm 2.3}. Pick a vector
\[
\begin{bmatrix}
h\\g\end{bmatrix} \in \underset{n \in \Z_+}{\bigcap}T^n\clh.
\]
For each $n \in \Z_+$, there exist $\begin{bmatrix}h_n\\g_n
\end{bmatrix}\in \clh=(\clh\ominus\cld_T)\oplus\cld_T$ such that
\[
T^n\begin{bmatrix}
h_n\\g_n
\end{bmatrix}=\begin{bmatrix}
h\\g
\end{bmatrix}.
\]
As in the proof of Theorem \ref{thm 2.3} (specifically, see \eqref{eqn3} and \eqref{eqn4}), we have
\begin{equation}\label{eqn7}
h=Sh_1 \oplus Ag_1 = (S^2h_2 +SAg_2)\oplus ABg_2=\cdots,
\end{equation}
and
\begin{equation}\label{eqn8}
g=Bg_1=B^2g_2=\cdots.
\end{equation}			
We also have, as in \eqref{eqn: g_m}, that
\[
g_m\in\underset{n\geq 0}{\bigcap}B^n\cld_T,
\]
for all $m\geq 1$. As before, pick $m_0 \in \Z_+$ such that
\[
g_{m_0} \neq 0,
\]
and $g_m = 0$ for all $m < m_0$. For each $m>m_0$, we have
\[
g_{m_0}=B^{m-m_0}g_m.
\]
As $g_{m_0}\ne 0$, it follows that $g_m\ne 0$. Moreover, given that $B$ is a contraction, we have
\[
\|g_{m_0}\|=\|B^{m-m_0}g_m\|\leq \|g_m\|.
\]
We conclude that
\[
\sup \Big\{\frac{1}{\|g_m\|} : m\geq m_0\Big\} \leq \frac{1}{\|g_{m_0}\|} < \infty.
\]
We claim that
\[
Ag_n \raro 0 \text{ as }n \raro \infty.
\]
We proceed as follows: By applying $S^*,\,{S^*}^2,\ldots$ to \eqref{eqn7} (recall that $S^* A = 0$), we find
\begin{align*}
h_1=&Sh_2 \oplus Ag_2 = (S^2h_3+SAg_3) \oplus ABg_3 = \cdots
\\
h_2=&Sh_3 \oplus Ag_3 = (S^2h_4+SAg_4) \oplus ABg_4 = \cdots
\\
&\vdots
\end{align*}
By repeatedly applying the relation $S^*A = 0$ to the identities involving $h_1$, we conclude that
\[
\|h_1\|^2 = \|h_2\|^2+\|Ag_2\|^2 = \|Sh_3 \oplus Ag_3\|^2+\|Ag_2\|^2 = \|Sh_3\|^2+\|Ag_3\|^2+\|Ag_2\|^2=\ldots
\]
Continuing in this way, we arrive at
\[
\|h_1\|^2=\|h_n\|^2+\underset{m=2}{\overset{n}{\sum}}\|Ag_m\|^2\geq \underset{m=2}{\overset{n}{\sum}}\|Ag_m\|^2,
\]
for all $n \geq 2$, which implies that
\[
\underset{m=2}{\overset{\infty}{\sum}}\|Ag_m\|^2 \leq \|h_1\| ^2<\infty.
\]
As a result, $Ag_n \raro 0$, which proves the claim.

\noindent Since $\cld_T$ is a finite-dimensional subspace, $B$ is of finite rank. Therefore, the sequence $\Big\{B \Big(\frac{g_n}{\|g_n\|}\Big)\Big\}$ has a convergent subsequence, say $\Big\{B \Big(\frac{g_{n_k}}{\|g_{n_k}\|}\Big)\Big\}$. Therefore, there exists $g\in\cld_T$ such that
\begin{equation}\label{eqn: B g_nk}
B \Big(\frac{g_{n_k}}{\|g_{n_k}\|}\Big) \raro g,
\end{equation}
as $k \raro \infty$. We claim that
\[
g = 0.
\]
Assume, on the contrary, that $g \neq 0$. By \eqref{induction equation}, we know that $g_m=B^ng_{n+m}$ for all $m, n \geq 1$. In particular, we have
\[
g_{n_k-1} = Bg_{n_k},
\]
and hence
\[
\frac{g_{n_k-1}}{\|g_{n_k}\|} = B \Big(\frac{g_{n_k}}{\|g_{n_k}\|}\Big) \raro g.
\]
Therefore,
\[
\frac{1}{\|g_{n_k}\|} A\Big(g_{n_k-1}\Big) = A\Big(\frac{g_{n_k-1}}{\|g_{n_k}\|}\Big) \raro A g.
\]
On the other hand, since $Ag_n \raro 0$ as $n \raro \infty$, and $\sup \Big\{\frac{1}{\|g_m\|} : m\geq m_0\Big\} < \infty$, we conclude that
\[
Ag=0.
\]
As already pointed out, we have $g_{n_k-1}\in\underset{n\geq 0}{\cap}B^n\cld_T$. Since $\underset{n\geq 0}{\cap}B^n\cld_T$ is a closed subspace (recall that $\cld_T$ is a finite-dimensional subspace of $\clh$), by \eqref{eqn: B g_nk}, it follows that
\[
B^mg \in\underset{n\geq 0}{\cap}B^n\cld_T,
\]
for $m\geq 0$.  Also, observe that
\[
A \Big(\frac{g_{n_k-2}}{\|g_{n_k}\|}\Big) = AB \Big(\frac{g_{n_k-1}}{\|g_{n_k}\|}\Big) \raro AB g.
\]
Since $Ag_n\rightarrow 0$ and $\left\{\frac{1}{\|g_{n_k}\|}\right\}$ is bounded above, we conclude that
\[
ABg=0.
\]
Continuing in this way, we get
\[
AB^mg=0,
\]
for every $m\in \Z_+$. Since $B$ is a finite-rank operator, there exists a polynomial, say
\[
p = \prod_{i=1}^m (z - \lambda_i),
\]
for some scalars $\{\lambda_i\}_{i=1}^m$, such that
\[
p(B)=(B-\lambda_1 I_{\cld_{T}})\cdots(B-\lambda_m I_{\cld_{T}}) = 0.
\]
In particular, $p(B)g=0$. This gives either
\[
(B-\lambda_2 I_{\cld_{T}})\cdots(B-\lambda_m I_{\cld_{T}})g\ne 0,
\]
or
\[
(B-\lambda_2I_{\cld_{T}})\cdots(B-\lambda_m I_{\cld_{T}})g=0.
\]
Equivalently, either $(B-\lambda_2 I_{\cld_{T}})\cdots(B-\lambda_m I_{\cld_{T}})g$ is an eigenvector of $B$, or it is the zero vector. Now, if $(B-\lambda_2 I_{\cld_{T}})\cdots(B-\lambda_m I_{\cld_{T}})g=0$, then either $(B-\lambda_3 I_{\cld_{T}})\cdots(B-\lambda_m I_{\cld_{T}})g$ is an eigenvector of $B$, or it is zero. Continuing in this manner, we conclude that there exists a polynomial $q$ (which may be the constant polynomial $1$) such that $q(B)g$ is an eigenvector of $B$. Since $AB^mg=0$ for all $m\in \Z_+$, we have
\[
Aq(B)g=0,
\]
which implies that $q(B)g$ is an eigenvector of $B$ that also lies in the kernel of $A$. This contradicts the fact that there is no eigenvector of $B$ in $\ker A$. Thus, we conclude the proof of the claim that $g=0$. It follows that
\[
(A^*A+B^*B)\Big(\frac{g_{n_k}}{\|g_{n_k}\|}\Big) \raro 0.
\]
However, $A^*A+B^*B$ is invertible and $\{\frac{g_{n_k}}{\|g_{n_k}\|}\}$ is an orthonormal sequence. Therefore, $(A^*A+B^*B)(\frac{g_{n_k}}{\|g_{n_k}\|})$ cannot converge to zero. Hence, $T$ is analytic. By \eqref{Shimorin thm}, it follows that $T$ is unitary equivalent to an analytic shift.
\end{proof}

In Example \ref{eg: finite-isome}, we will illustrate the above result with a concrete example.

\section{$n$-finite operators}\label{sec: fin rank self com}

Let $T \in \clb(\clh)$ be a hyponormal contraction. Since
\[
[T^*, T] = D^2_{T^*} - D^2_T \leq D^2_{T^*},
\]
it follows that
\[
\text{ran} [T^*, T] \subseteq \cld_{T^*}.
\]
In other words, the range of the self commutator of $T$ is contained in the defect space $\cld_{T^*}$. Since we also know that $\cld_T \subseteq \cld_{T^*}$, it is natural to study those hyponormal contractions $T$ for which the self-commutator is of finite-rank and its range is contained in $\cld_{T^*} \ominus \cld_T$. From this perspective, the notion of $n$-finite operators emerges as a natural and significant class of operators to consider.

We first prove a result of independent interest: If $T$ is  $n$-finite, then in particular
\[
\text{dim} (\cld_{T^*} \ominus \cld_{T}) = n.
\]

\begin{proposition}\label{prop: n finite}
Let $T \in \clb(\clh)$ be a hyponormal contraction. Write $[T^*, T]$ as
\[
[T^*, T] = \sum_{j=1}^n \alpha_j e_j \otimes e_j,
\]
for some orthonormal set $\{e_j\}_{j=1}^n \subseteq \cld_{T^*} \ominus \cld_T$ and scalars $\{\alpha_j\}_{j=1}^n \subseteq (0,1]$. Then $\{e_j\}_{j=1}^n$ is an orthonormal basis for $\cld_{T^*} \ominus
 \cld_T$.
\end{proposition}
\begin{proof}
Let $h\in \cld_{T^*}\ominus\cld_T$. There exist a sequence $\{g_m\} \subseteq \clh$ such that
\[
h = \lim\limits_{m\to \infty} D_{T^*}^2g_m.
\]
For each $m \geq 1$, define
\[
h_m = D_{T^*}^2g_m.
\]
Decompose $\cld_{T^*}$ as $\cld_{T^*}=\cld_{T}\oplus(\cld_{T^*}\ominus\cld_{T})$. Since $h_m\in \cld_{T^*}$, there exist $u_m\in \cld_{T}$ and $v_m\in \cld_{T^*}\ominus\cld_{T}$ such that
\[
h_m = u_m \oplus v_m.
\]
Write $h_m-h = u_m \oplus (v_m-h)$. As $h_m \raro h$, it follows that
\[
\lim\limits_{m\to \infty}v_m=h,
\]
and
\[
\lim\limits_{m\to \infty}u_m=0.
\]
Since
\[
[T^*, T] = D_{T^*}^2-D_T^2,
\]
it follows that
\[
D_{T^*}^2-D_T^2 = \sum\limits_{i=1}^{n}\alpha_ie_i\otimes e_i.
\]
Therefore, for each $m \geq 1$, we have
\[
D_{T^*}^2g_m-D_T^2g_m = \sum\limits_{i=1}^{n}\alpha_i\langle g_m,e_i\rangle e_i.
\]
As $D^2_{T^*} g_m = h_m$, we have
\begin{align*}
h_m-D_T^2g_m=	u_m+v_m-D_T^2g_m=\sum\limits_{i=1}^{n}\alpha_i\langle g_m,e_i\rangle e_i.
\end{align*}
Since $e_i,v_m\in \cld_{T^*}\ominus\cld_{T}$ for all $i = 1, \ldots, n$, we have
\[
u_m-D_T^2g_m\in \cld_{T^*}\ominus\cld_{T}.
\]
On the other hand, we know that $u_m-D_T^2g_m\in \cld_{T}$. Thus $u_m-D_T^2g_m=0$, and hence
\[
v_m=\sum\limits_{i=1}^{n}\alpha_i\langle g_m,e_i\rangle e_i\in \text{span}\{e_1, \ldots ,e_n\}.
\]
Since $h=\lim\limits_{m\to \infty}v_m$, it follows that
\[
h\in \text{span}\{e_1, \ldots, e_n\},
\]
and hence $\cld_{T^*}\ominus\cld_{T}=\text{span}\{e_1, \ldots, e_n\}$. This completes the proof of the result.
\end{proof}

In particular, if $T \in \clb(\clh)$ is an $n$-finite operator, then there exists an orthonormal basis $\{e_j\}_{j=1}^n $ for $\cld_{T^*} \ominus \cld_T$  and scalars $\{\alpha_j\}_{j=1}^n \subseteq (0,1]$ such that
\[
[T^*, T] = \sum_{j=1}^n \alpha_j e_j \otimes e_j.
\]
We are now ready to present the model for $n$-finite operators. Recall the notation: For $\{\alpha_i\}_{i=1}^m \subseteq (0,1]$, denote by $D_{\{\alpha_i\}_{i=1}^m} \in\clb(\mathbb{C}^m)$ the $m \times m$ diagonal matrix
\[
D_{\{\alpha_i\}_{i=1}^m} = \text{diag} \{\sqrt{1-\alpha_1}, \ldots, \sqrt{1-\alpha_m}\}.
\]
Moreover, given $n \geq m$, we define $\tilde{D}_{\{\alpha_i\}_{i=1}^m}: \mathbb{C}^m \raro \mathbb{C}^n$ by
\[
\tilde{D}_{\{\alpha_i\}_{i=1}^m} = \begin{bmatrix} D_{\{\alpha_i\}_{i=1}^m} \\0_{\mathbb{C}^{n-m}}\end{bmatrix}.
\]

\begin{theorem}\label{main result2}
Let $T \in \clb(\clh)$. Then $T$ is  $n$-finite, with
\[
[T^*, T] = \sum_{j=1}^n \alpha_j e_j \otimes e_j,
\]
for some orthonormal set $\{e_j\}_{j=1}^n \subseteq \cld_{T^*} \ominus \cld_T$ and scalars $\{\alpha_j\}_{j=1}^n \subseteq (0,1]$, if and only if
$T \cong X_1$ or $T \cong X_2$, where
\[
X_1=\begin{bmatrix}
M_z\otimes I&0\\
0&N
\end{bmatrix} \in \clb((H^2(\bd)\otimes\mathbb{C}^n)\oplus \clh_0),
\]
and
\[
X_2=\begin{bmatrix}
M_z\otimes I&P_{\mathbb{C}}\otimes \tilde{D}_{\{\alpha_i\}_{i=1}^m} & 0
\\
0&M_z^*\otimes {D}_{\{\alpha_i\}_{i=1}^m} & 0
\\
0&0&N
\end{bmatrix} \in \clb((H^2(\bd)\otimes\mathbb{C}^n)\oplus(H^2(\bd)\otimes\mathbb{C}^m)\oplus \clh_0),
\]
and $N$ is a normal operator acting on a closed subspace $\clh_0\subseteq\clh$.
\end{theorem}
\begin{proof}
Suppose $T$ is $n$-finite. From Proposition \ref{prop: n finite} there exists an orthonormal basis $\{e_j\}_{j=1}^n $ for $ \cld_{T^*} \ominus \cld_T$  and scalars $\{\alpha_j\}_{j=1}^n \subseteq (0,1]$ such that
\[
[T^*, T] = \sum_{j=1}^n \alpha_j e_j \otimes e_j.
\]
In view of Proposition \ref{matrix rep}, we decompose $T$ as
\begin{equation}\label{eqn 10}
T=\begin{bmatrix}
S&A\\
0&B
\end{bmatrix},
\end{equation}
on $\clh=(\clh\ominus\cld_T)\oplus\cld_T$, where $S \in \clb(\clh\ominus\cld_T)$ is an isometry and $S^*A=0$. From the construction of $S$ in Proposition \ref{matrix rep}, we know that $S = T|_{\clh\ominus \cld_T}$ and $\operatorname{ran} S = \clh \ominus \cld_{T^*}$. This implies that
\[
\ker S^* = (\text{ran} S)^\perp = (\clh\ominus \cld_T) \ominus (\clh \ominus \cld_{T^*}) =\cld_{T^*} \ominus  \cld_T.
\]
It follows that
\[
P_{\ker S^*} = \underset{i=1}{\overset{n}{\sum}}e_i\otimes e_i.
\]
Also, we have
\[
[T^*, T] =
\begin{bmatrix}
P_{\ker S^*}-AA^*&-AB^*
\\
-BA^*&A^*A+B^*B-BB^*
\end{bmatrix},
\]
on $\clh=(\clh\ominus\cld_T)\oplus\cld_T$. As $\{e_j\}_{j=1}^n \subseteq \cld_{T^*} \ominus \cld_T$ , it follows that
\[
[T^*, T]|_{\cld_T}=0.
\]
Therefore, we have $AB^*=0$ and
\[
A^*A+B^*B-BB^*=0,
\]
and
\[
P_{\ker S^*} - AA^* = [T^*, T].
\]
The final identity implies that
\[
AA^*=\sum\limits_{i=1}^{n}(1-\alpha_i) e_i\otimes e_i.
\]
In particular, $A:\cld_T\to \cld_{T^*}\ominus\cld_T$ is a finite-rank operator. There exist $\{h_i\}_{i=1}^n \subseteq \cld_T$ such that
\begin{equation}\label{eqn: A = 1 to n}
A=\sum\limits_{i=1}^{n}e_i\otimes h_i.
\end{equation}
This implies
\[
A A^* =\sum\limits_{i,j=1}^{n}\langle h_j,h_i\rangle e_i\otimes e_j.
\]
By comparing the two expressions for $AA^*$, we find that
\[
\sum\limits_{i=1}^{n}(1-\alpha_i) e_i\otimes e_i=\sum\limits_{i,j=1}^{n}\langle h_j,h_i\rangle e_i\otimes e_j,
\]
and hence
\begin{equation}\label{eqn:  hij delta}
\langle h_i,h_j\rangle=\delta_{ij}(1-\alpha_i),
\end{equation}
for all $i, j = 1, \ldots, n$. Note that $\alpha_i \in (0,1]$ for all $i=1, \ldots, n$.

\noindent \textsf{Case I:} Let $\alpha_i=1$ for all $i = 1, \ldots, n$. Then $A=0$. Write $S=U_0\oplus S_0$ on $\clh \ominus \cld_T = \clw_1 \oplus \clw_2$ in view of the von Neumann-Wold decomposition theorem, where $U_0$ is a unitary and $S_0$ is a shift operator. Then
\[
T=\begin{bmatrix}
U_0&0&0\\
0&S_0&0\\
0&0&B\\
\end{bmatrix},
\]
with respect to the decomposition $(\clh \ominus \cld_T) \oplus \cld_T = \clw_1\oplus \clw_2\oplus \cld_T$. It follows that
\[
[T^*, T] =\begin{bmatrix}
0&0&0\\
0&P_{\ker S^*_0}&0\\
0&0&B^*B-BB^*
\end{bmatrix}.
\]
Again, since
\[
[T^*, T]|_{\cld_T}=0,
\]
we have $[B^*,B] = 0$, that is, $B$ is a normal operator. Set $\clh_0= \clw_1\oplus\cld_T$, and define $U:(H^2(\mathbb{D})\otimes \mathbb{C}^n)\oplus \clh_0\to \clh = \clw_2\oplus\clh_0$ by
\[
U(z^mf_i+y)=S_0^me_i+ y,
\]
for all $m \in \Z_+$ and $i=1, \ldots, n$, where $\{f_i:i=1, \ldots, n\}$ is the standard orthonormal basis for $\mathbb{C}^n$. Since $\{S_0^me_i:m\in \Z_+, i=1, \ldots, n\}$ forms an orthonormal basis for $\clw_2$, it follows that $U$ is a unitary operator and
\[
U\begin{bmatrix}
M_z\otimes I&0\\
0&N
\end{bmatrix}=TU,
\]
where $N=U_0\oplus B$ is the normal operator on $\clh_0$.
	
\noindent \textsf{Case II:} Suppose there exist $i \in \{1, \ldots, n\}$ such that $\alpha_i<1$ (recall that $\alpha_j \neq 0$ for all $j=1, \ldots, n$). Without loss of generality, assume that
\[
\alpha_1\leq\alpha_2\leq\cdots\leq\alpha_n,
\]
and suppose
\[
m :=\max\{j:\alpha_j<1\}.
\]
Note that $1\leq m\leq n$. Since \eqref{eqn:  hij delta} implies $\langle h_i,h_j\rangle=\delta_{ij}(1-\alpha_i)$ for all $i, j = 1, \ldots, n$, it follows that $h_j=0$ for all $j>m$. Hence, the representation of $A$ in \eqref{eqn: A = 1 to n} simplifies to
\begin{equation*}
A=\sum_{i=1}^{m}e_i\otimes h_i.
\end{equation*}
Therefore,
\[
AB^*=\sum\limits_{i=1}^{m}e_i\otimes Bh_i.
\]
At this point, we use the information that $AB^* = 0$ and $A^*A + [B^*, B] = 0$. First, we have
\[
Bh_j=0,
\]
for all $j = 1, \ldots, m$, and on the other hand,
\[
[B, B^*] = A^*A = (\sum\limits_{j=1}^{m}h_j\otimes e_j)(\sum\limits_{j=1}^{m}e_j\otimes h_j)=\sum\limits_{j=1}^{m}h_j\otimes h_j.
\]
We have
\begin{align}\label{hypB}
BB^* - B^* B = \sum\limits_{j=1}^{m} h_j \otimes h_j.
\end{align}
Fix $i \in \{1, \ldots, m\}$. Since $\{h_1, \ldots, h_m\}\subseteq \ker B$, it follows that $B^* B h_i = 0$, and hence
\[
B B^* h_i = (\sum\limits_{j=1}^{m} h_j \otimes h_j)h_i.
\]
Then, \eqref{eqn:  hij delta} yields
\[
BB^*h_i = (1-\alpha_i) h_i.
\]
Moreover, for each $n \geq 2$, the identity \eqref{hypB} implies
\[
\begin{split}
BB^{*n} h_j - B^* B B^{*(n-1)}h_j = \sum\limits_{j=1}^{m} \langle B^{*(n-1)} h_j, h_j \rangle h_j = \sum\limits_{j=1}^{m} \langle h_j, B^{n-1} h_j \rangle h_j = 0,
\end{split}
\]
that is,
\[
BB^{*n} h_i = B^* (B B^{*(n-1)}h_i),
\]
for all $n \geq 2$. In particular,
\[
B B^{*2} h_i = B^* (B B^* h_i) = (1 - \alpha_i) B^* h_i,
\]
and, continuing in this way, it follows by induction that, in general,
\begin{align}\label{eqn: BB*}
BB^{*n}h_i=(1-\alpha_i)B^{*^{}n-1}h_i,
\end{align}
for all $n \geq 1$. We claim that $\clb_i$ is an orthonormal set, where
\[
\clb_i= \Big\{\frac{1}{\|h_i\|^{m+1}}B^{*m}h_i: m \in \Z_+\Big\}.
\]
For each $p>q \geq 0$, by repeatedly applying \eqref{eqn: BB*}, we find that
\begin{align*}
\langle B^{*q}h_i,B^{*p}h_i\rangle&=\langle h_i,B^q B^{*p}h_i\rangle\\
&=\langle h_i,B^{q-1}(1-\alpha_i)B^{*^{p-1}}h_i\rangle
\\
&\vdots
\\
&=\langle h_i,(1-\alpha_i)^q B^{*^{p-q}}h_i\rangle
\\
&=\langle B^{p-q}h_i,(1-\alpha_i)^q h_i\rangle
\\
&=0.
\end{align*}
On the other hand, for each $p \geq 1$, we have
\begin{align*}
\|B^{*p}h_i\|^2 & = \langle B^{*p}h_i,B^{*p}h_i\rangle
\\
&=\langle h_i,B^p B^{*p}h_i\rangle
\\
&=\langle h_i,(1-\alpha_i)^p h_i\rangle
\\
&=(1-\alpha_i)^p\|h_i\|^2
\\
&=(\langle h_i,h_i\rangle)^p\|h_i\|^2
\\
&=\|h_i\|^{2(p+1)}.
\end{align*}
In the above, we have used the fact that $\langle h_s,h_t\rangle=\delta_{st}(1-\alpha_s)$ for all $s, t = 1, \ldots,n$ (see \eqref{eqn:  hij delta}). This implies that
\[
\|B^{*p}h_i\|=\|h_i\|^{p+1},
\]
for every $p \geq 1$, thus completing the proof of the claim. Define
\[
\clm_i=\overline{\text{span}}\clb_i.
\]
By construction (or see \eqref{eqn: BB*}), $\clm_i$ reduces $B$. Moreover, for each $n \geq 1$, we compute:
\[
\begin{split}
B|_{\clm_i} \Big(\frac{1}{\|h_i\|^{n+1}} B^{*n} h_i\Big) = \frac{(1 - \alpha_i)}{\|h_i\|^{n+1}} B^{*{(n-1)}} h_i = \frac{(1 - \alpha_i)}{\|h_i\|} \Big(\frac{1}{\|h_i\|^{n}} B^{*(n-1)} h_i\Big).
\end{split}
\]
Since $\|h_i\| = \sqrt{1 - \alpha_i}$, it follows that
\[
B|_{\clm_i} \Big(\frac{1}{\|h_i\|^{n+1}} B^{*n} h_i\Big) = \sqrt{1 - \alpha_i} \Big(\frac{1}{\|h_i\|^{n}} B^{*(n-1)} h_i\Big),
\]
and hence, $B|_{\clm_i}$ is $\sqrt{1 - \alpha_i}$ times $S_i^*$, where $S_i$ is the unilateral shift of multiplicity one with respect to the orthonormal basis $\clb_i$ for $1\leq i\leq m$. In other words, we have
\[
B|_{\clm_i} \cong \sqrt{1 - \alpha_i}S_i^*.
\]
Define
\[
\clm=\oplus_{i=1}^{m} \clm_i.
\]
Then $\clm$ reduces $B$ and
\[
B|_\clm \cong (1-\alpha_1)^{1/2}S_1^* \oplus\cdots\oplus(1-\alpha_m)^{1/2}S_m^*.
\]
Recall that $A: \cld_T \raro \cld_{T^*} \ominus \cld_T$ is given by $A = \sum_{i=1}^{m} e_i \otimes h_i$. Since $h_i \in \clm$ for all $i=1, \ldots, n$, it follows that
\[
A|_{\cld_T\ominus\clm}=0.
\]
Recall moreover that $T = \begin{bmatrix} S & A \\ 0 & B \end{bmatrix}$ on $(\clh \ominus \cld_T) \oplus \cld_T$. Write
\[
B = \begin{bmatrix} B|_{\clm} & 0 \\ 0 & C \end{bmatrix},
\]
on $\clm \oplus (\cld_T \ominus \clm)$. As $A|_{\cld_T\ominus\clm}=0$, it follows that
\[
T=\begin{bmatrix}
U_0&0&0&0\\
0&S_0&A&0\\
0&0&B|_\clm&0\\
0&0&0&C
\end{bmatrix},
\]
where $S= \begin{bmatrix} U_0 & 0 \\ 0 & S_0 \end{bmatrix}$ on $\clh \ominus \cld_T = \clw_1 \oplus \clw_2$ is the von Neumann-Wold decomposition, with $U_0 \in \clb(\clw_1)$ being a unitary and $S_0 \in \clb(\clw_2)$ a unilateral shift operator. We have
\[
[T^*, T] = \begin{bmatrix}
0&0&0&0\\
0&P_{\ker S^*_0}-AA^*&0&0\\
0&0&0&0\\
0&0&0&C^*C-CC^*
\end{bmatrix},
\]
on $\clw_1\oplus \clw_2\oplus \clm\oplus(\cld_T\ominus\clm)$. Since $[T^*, T]|_{\cld_T}=0$, it follows that $C$ is a normal operator. We also know that $AA^*=\sum\limits_{i=1}^{n}(1-\alpha_i) e_i\otimes e_i.$ (see the identity preceding \eqref{eqn: A = 1 to n}). Since $\alpha_i=1 $ for $m+1\leq i\leq n$, and $P_{\ker S^*_0} = AA^*+ [T^*, T] $, we have
\[
P_{\ker S^*_0} = \sum\limits_{i=1}^{m}(1-\alpha_i)e_i\otimes e_i+\sum\limits_{i=1}^{n}\alpha_ie_i\otimes e_i = \sum\limits_{i=1}^{m}e_i\otimes e_i+\sum\limits_{i=m+1}^{n}\alpha_ie_i\otimes e_i,
\]
and hence
\[
P_{\ker S^*_0} = \sum\limits_{i=1}^{n}e_i\otimes e_i.
\]
Suppose $\{f_i: i = 1, \ldots, n\}$ and $\{g_j: j = 1, \ldots, m\}$ are the standard orthonormal bases for $\mathbb{C}^n$ and $\mathbb{C}^m$, respectively. Set
\[
\clh_0=W_1\oplus(\cld_T\ominus\clm),
\]
and define $U:(H^2(\mathbb{D})\otimes \mathbb{C}^n)\oplus(H^2(\mathbb{D})\otimes \mathbb{C}^m)\oplus \clh_0\to \clh$ by
\[
U(z^sf_i+ z^rg_j+y)=S_0^se_i+ \frac{1}{(1-\alpha_j)^{\frac{r+1}{2}}}B^{*r}h_j + y,
\]
for all $s,r \in \Z_+$, $i = 1, \ldots, n$, and $j = 1, \ldots, m$. Since $\{S_0^s e_i:s\in \Z_+, i = 1, \ldots, n\}$ forms an orthonormal basis for $\clw_2$, and
\[
\left\{\frac{1}{(1-\alpha_j)^{\frac{r+1}{2}}} B^{*r}h_j :r \in \Z_+, j = 1, \ldots, m\right\},
\]
forms an orthonormal basis for $\clm$, we conclude that $U$ is a unitary. We claim that
\[
U^*TU=\begin{bmatrix}
M_z\otimes I & P_{\mathbb{C}} \otimes \tilde{D}_{\{\alpha_i\}_{i=1}^m} & 0\\
0& M_z^* \otimes {D}_{\{\alpha_i\}_{i=1}^m}&0\\
0&0&N
\end{bmatrix},
\]
where $N$ is the normal operator $U_0\oplus C$. To prove our claim, first we observe that $AB^*=0$. Moreover,
\begin{align*}
TU(z^sf_i+ z^rg_j+y)= \begin{cases}
S_0^{s+1}e_i+ \frac{\sqrt{1-\alpha_j}}{(1-\alpha_j)^{\frac{r}{2}}} B^{*(r-1)}h_j + Ny &\text{ if }r\ge1\\
S_0^{s+1}e_i+\sqrt{1-\alpha_j}e_j+Ny&\text{ if }r=0.
\end{cases}
\end{align*}
For each $r\geq 1$, we compute
\begin{align*}
U X_2 (z^sf_i+ z^rg_j+y) & = U\begin{bmatrix}
M_z\otimes I & P_{\mathbb{C}} \otimes \tilde{D}_{\{\alpha_i\}_{i=1}^m} &0\\ 0 & M_z^* \otimes {D}_{\{\alpha_i\}_{i=1}^m} & 0\\ 0 & 0 & N \end{bmatrix}
\begin{bmatrix}
z^sf_i\\z^rg_j\\ y\end{bmatrix}
\\
& = U(z^{s+1}f_i+\sqrt{1-\alpha_j}z^{r-1}g_j+Ny)
\\
& = S_0^{s+1}e_i + \frac{\sqrt{1-\alpha_j}}{(1-\alpha_j)^{\frac{r}{2}}} B^{*(r-1)}h_j +Ny.
\end{align*}
In the case of $r=0$, we have
\begin{align*}
UX_2(z^sf_i+ g_j+y) & = U\begin{bmatrix}
M_z\otimes I & P_{\mathbb{C}} \otimes \tilde{D}_{\{\alpha_i\}_{i=1}^m} & 0\\0& M_z^* \otimes {D}_{\{\alpha_i\}_{i=1}^m} & 0\\0&0&N
\end{bmatrix}
\begin{bmatrix}
z^sf_i\\
g_j\\
y
\end{bmatrix}
\\
& = U(z^{s+1}f_i+\sqrt{1-\alpha_j}f_j+0+Ny)
\\
& = S_0^{s+1}e_i+\sqrt{1-\alpha_j}e_j+Ny.
\end{align*}
Therefore, we conclude that $U^*TU=X_2$.
	
\noindent Conversely, assume that there exist a unitary $U$ such that
\[
U^*TU = \begin{bmatrix}
M_z\otimes I & P_{\mathbb{C}} \otimes \tilde{D}_{\{\alpha_i\}_{i=1}^m} & 0\\0& M_z^* \otimes {D}_{\{\alpha_i\}_{i=1}^m} & 0\\0&0&N
\end{bmatrix},\text{ or }U^*TU=\begin{bmatrix}
M_z\otimes I&0\\
0&N
\end{bmatrix}
\]
where $N$ a normal operator. A routine computation shows that $[T^*,T]=\sum\limits_{i=1}^{n}\alpha_i e_i\otimes e_i$, where $\{e_i:1\leq i\leq n\}\subseteq \cld_{T^*}\ominus\cld_T$ is an orthonormal set and $\alpha_i\in (0,1]$.
\end{proof}

\section{Examples}\label{sec: eg}

The aim of this section is to present several examples that illustrate the notions introduced and the results established in this paper. We begin with an example of an operator that satisfies the conditions of Theorem \ref{thm 2.3}.

\begin{example}\label{eg 2.4}
Let $S\in \clb(\clh)$ be a unilateral shift of infinite multiplicity: 
\[
\text{dim} (\ker S^*) = \infty.
\]
Fix a natural number $n$. Let $\clm$ be a finite-dimensional subspace of $\clh$ with an orthonormal basis $\{x_1, \ldots,x_n\}$. Define $S_1:\clm\to \clm$ by
\[
S_1x_i=\begin{cases}
0&\text{ if }i=1\\
x_{i-1}&\text{ if } i =2, \ldots, n.
\end{cases}
\]
Fix $\lambda \in \D$. Define $B\in\clb(\clh)$ by
\[
B|_\clm=\lambda S_1,
\]
and
\[
B|_{\clh\ominus\clm}=0.
\]
In other words,
\[
B=\begin{bmatrix}
\lambda S_1 & 0 \\
0& 0
\end{bmatrix},
\]
on $\clm \oplus \clm^\perp$. Clearly, $B$ is a nilpotent operator, and, in particular, $B$ is analytic. Let $A\in\clb(\clh)$ be a partial isometry with initial space $(\clh\ominus\clm) \oplus \text{span}\{x_n\}$ and final space contained in $\ker S^*$. Therefore,
\[
A^*A=P_{(\clh\ominus\clm)\oplus\,\text{span}\{x_n\}}.
\]	
Also, we have $\text{ran} A\subseteq \ker S^*$, or equivalently,
\[
S^* A = 0.
\]
Finally, define $T \in \clb(\clh\oplus\clh)$ by
\[
T=\begin{bmatrix}
S&A\\
0&B
\end{bmatrix}.
\]
It follows that
\[
TT^*=\begin{bmatrix}
P_{\text{ran} S}+P_{\text{ran} A}&0\\
0& |\lambda|^2P_{span\{x_2,\ldots,x_n\}}
\end{bmatrix},
\]
and
\[
T^*T=\begin{bmatrix}
I&0\\
0& A^*A + B^* B
\end{bmatrix}.
\]
Moreover, as
\[
A^*A + B^* B = P_{(\clh\ominus\clm)\oplus\langle \{x_n\}\rangle}+|\lambda|^2P_{span\{x_1,\ldots,x_{n-1}\}},
\]
we conclude that $A^*A + B^* B$ is an invertible operator. Clearly, $T$ is a hyponormal contraction. By Theorem \ref{thm 2.3}, it follows that $T$ is unitarily equivalent to a shift.
\end{example}

Recall from \eqref{Shimorin thm} that analytic and left-invertible operators are precisely analytic shifts. In the above, we have exhibited certain examples of upper triangular block operator matrices that are unitarily equivalent to analytic shifts. However, an upper triangular block operator with a unilateral shift on the north-west entry can also produce an example of a non-analytic but left-invertible operator:

\begin{example}
We follow the construction of Proposition \ref{proposition 1}. Let $S \in \clb(\clh)$ be a unilateral shift of finite multiplicity and let $X \in \clb(\clh)$. Suppose
\[
(\ker S^*)^\perp = \ker X.
\]
Then
\[
T:=\begin{bmatrix}
S&X\\
0&S^*
\end{bmatrix},
\]
is not analytic but left-invertible. To see this, for each $n \geq 1$, we compute
\[
T^n=\begin{bmatrix}
S^n&S^{n-1}X+XS^{*(n-1)}
\\
0&S^{*n}
\end{bmatrix}.
\]
Fix a nonzero $g\in \clh$ and set $x=Xg$. We claim that
\[
\begin{bmatrix}
x\\S^*g
\end{bmatrix} \in \bigcap_{n=0}^\infty T^n (\clh \oplus \clh).
\]
Fix a natural number $n$. Since $S^*\clh=\clh$, there exits
\[
g_n\in (\ker S^*)^\perp = (\text{ran} X^*)^\perp=\ker X,
\]
such that $g=S^{*(n-1)}g_n$. Then $Xg_n=0$, and $T^n\begin{bmatrix} 0\\g_n \end{bmatrix}=\begin{bmatrix}x\\S^*g \end{bmatrix}$. This proves that $T$ is not analytic.
\end{example}

We also have example related to Toeplitz operators:

\begin{example}
Pick inner functions $\theta, \varphi \in H^\infty(\D)$ and define the model space $\clq_\theta$ by
\[
\clq_\theta = H^2(\D) \ominus \theta H^2(\D).
\]
Assume that (for instance, if $\theta$ is an infinite Blaschke product)
\[
\text{dim} \clq_\theta = \infty.
\]
Define $A=T_{z\theta}$ and $B=T_\varphi$ on $H^2(\D)$. Clearly,
\[
A^*A+B^*B=2I,
\]
is invertible. Note that $\{\theta z^n\}_{n \in \Z_+}$ is an orthonormal basis for $\theta H^2(\D)$. Let
\[
\{e_{i,j}:i \in \mathbb{N}, j \in \Z_+\},
\]
be an orthonormal basis for $\clq_\theta$. Define $S\in\clb(H^2(\D))$ by
\[
S(\theta z^m)=e_{m+1,0},
\]
for all $m \in \Z_+$, and
\[
S(e_{n,m})=e_{n,m+1},
\]
for $n\ge1$ and $m \in \Z_+$. Clearly, $S$ is a unilateral shift on $H^2(\D)$ with
\[
\ker S^* = \theta H^2(\D).
\]
Moreover,
\[
\text{ran} A = z\theta H^2(\D)\subseteq\theta H^2(\D) = \ker S^*.
\]
By Theorem \ref{thm 2.3}, we conclude that $T=\begin{bmatrix} S&A\\0&B \end{bmatrix}$ is unitarily equivalent to an analytic shift.
\end{example}

We now present an example of a finite-isometry, which also serves to illustrate Theorem \ref{main thm}.

\begin{example}\label{eg: finite-isome}
Define linear operators $A:\mathbb{C}^2\rightarrow H^2(\bd)$ and $B:\mathbb{C}^2\rightarrow\mathbb{C}^2$ by
\[
A(x_1,x_2)=x_1,
\]
and
\[
B(x_1,x_2)=\left(\frac{x_2}{2},0\right),
\]
respectively, for all $(x_1,x_2) \in \mathbb{C}^2$. Define $T$ on $H^2(\bd)\oplus H^2(\bd)\oplus\mathbb{C}^2$ by
\begin{align*}
T = \begin{bmatrix}
M_z&0&0\\
0&M_z&A\\
0&0&B
\end{bmatrix}.
\end{align*}
Then
\[
T^*T-TT^* = \begin{bmatrix}
P_{\mathbb{C}}&0&0\\
0&0&0\\
0&0&C
\end{bmatrix},
\]
where $C\in\clb(\mathbb{C}^2)$ is defined by
\[
C(x_1,x_2)=\frac{1}{4}(3x_1,x_2),
\]
for all $(x_1,x_2) \in \mathbb{C}^2$. It is easy to see that $T$ is a finite isometry. We claim that $T$ is pure. Suppose, on the contrary, that there exists a reducing subspace $\clm\subseteq H^2(\bd)\oplus H^2(\bd)\oplus\mathbb{C}^2$ for $T$ such that $T|_{M}$ is a normal operator. Then
\[
\clm\subseteq \text{ker}(T^*T-TT^*)=zH^2(\bd)\oplus H^2(\bd)\oplus\{0\}.
\]
Let $zf+g+0\in\clm$ for $f,g\in H^2(\bd)$. Then
\[
T^*(zf+g+0)\in\clm.
\]
which implies that
\[
f+M_z^*g+A^*g\in zH^2(\bd)\oplus H^2(\bd)\oplus\{0\}.
\]
We have $A^*g=0$, hence
\[
\langle g,1\rangle=\langle g,A(1,0)\rangle=\langle A^*g,(1,0)\rangle=0.
\]
As a result, we conclude that $f,g\in zH^2(\bd)$. Repeating the same steps with ${T^*}^n$ for $n>1$, we obtain $\{zf, f, M_z^*f, \ldots\} \subseteq zH^2(\bd)$ and $\{g, M_z^*g, {M_z^*}^2g, \ldots\} \subseteq zH^2(\bd)$. This is only possible if
\[
f = g=0.
\]
Thus, $\clm=\{0\}$, and as a result, $T$ is a pure hyponormal operator. By Theorem \ref{main thm}, it follows that $T$ is unitarily equivalent to an analytic shift.
\end{example}

We conclude the paper with a simple observation regarding operators whose self-commutator is an orthogonal projection. Let $T\in\clb(\clh)$ be a contraction such that $[T^*,T]$ a finite-rank orthogonal projection. There exist orthonormal vectors $\{e_i\}_{i=1}^n \subseteq \clh$ such that
\[
T^*T-TT^*=\sum_{i=1}^n e_i\otimes e_i.
\]
This implies
\[
TT^*+\sum_{i=1}^n e_i\otimes e_i=T^*T.
\]
Since $T$ is a contraction, it follows that $T^*T\leq I$, and hence
\[
TT^*+\sum_{i=1}^n e_i\otimes e_i \leq I.
\]
As a result, $\|T^*e_i\|^2+1\leq 1$, and hence $T^*e_i=0$ for all $i=1, \ldots, n$. This implies $(I-T^*T)e_i=0$, and $D_{T^*}^2e_i = e_i$ for $i= 1, \ldots, n$. We conclude that $\{e_i:1\leq i\leq n\}\subseteq \cld_{T^*}\ominus\cld_T$. Theorem \ref{main result2} now implies that $T$ is unitarily equivalent to $S\oplus A$, where $S$ is a unilateral shift and $A$ is a normal operator.

\bigskip

\noindent\textbf{Acknowledgement:} The research of the first named
author is supported by NBHM (National Board of Higher Mathematics, India) Ph.D.
fellowship No. 0203/13(47)/2021-R\&D-II/13177. The research of the second named author is supported by the FRS grant, IIT (ISM) Dhanbad (File no. FRS(218)/2024-2025/M\&C). The research of the third named author is supported in part by ANRF, Department of Science \& Technology (DST), Government of India (File No: ANRF/ARGM/2025/000130/MTR).

\end{document}